\documentclass[11pt]{amsart}
\usepackage{geometry}                % See geometry.pdf to learn the layout options. There are lots.
\usepackage{graphicx}
\usepackage{amssymb}
\usepackage{epstopdf}
\usepackage[sc]{mathpazo}
\usepackage{graphicx}
\usepackage{amssymb}
\usepackage{epstopdf}
\usepackage{mathrsfs}
\usepackage{amsmath,amscd}
\usepackage{mathtools}
\usepackage{MnSymbol}
\usepackage[nospace, noadjust]{cite}

\usepackage{verbatim}
\newcommand{\A}{\mathbb{A}}

\newcommand{\Q}{\mathbb{Q}}

\renewcommand{\P}{\mathbb{P}}

\newcommand{\ZZ}{\mathbb{Z}}

\newcommand{\p}{\mathfrak{p}}

\newcommand{\m}{\mathfrak{m}}

\renewcommand{\sp}{\textrm{Spec }}

\DeclareMathOperator{\Sing}{Sing}
\DeclareMathOperator{\codim}{codim}
\DeclareMathOperator{\depth}{depth}
\DeclareMathOperator{\newchar}{char}

\newtheorem{thm}{Theorem}[section]

\newtheorem*{theoA}{Theorem A}
\newtheorem*{theoB}{Theorem B}
\newtheorem*{qst}{Question}

\newtheorem{lm}[thm]{Lemma}
\newtheorem{prop}[thm]{Proposition}
\theoremstyle{definition}
\theoremstyle{definition}\newtheorem{defi}[thm]{Definition}
\theoremstyle{definition}\newtheorem{rmk}[thm]{Remark}
\newtheorem{setup}[thm]{Setup}

\title{Potential log canonical centers}
\author{Lorenzo Prelli}
\address{University of Washington, Department of Mathematics, Box 354350, Seattle, WA 98195-4350, USA}
\email{lorenzop@uw.edu}

\date{}                                           

\begin{document}

\begin{abstract}
	Given an ambient variety $X$ and a fixed subvariety $Z$ we give sufficient conditions for the existence of a boundary $\Delta$ such that $Z$ is a log canonical center for the pair $(X, \Delta)$. We also show that under some additional hypotheses $\Delta$ can be chosen such that $(X, \Delta)$ has log canonical singularities.
\end{abstract}

\maketitle

\tableofcontents

\section*{Introduction}

Log canonical centers are special subvarieties of a pair $(X,\Delta)$, and their behavior is related to the properties of log canonical singularities; they also provide a framework for inductive arguments in higher dimensional algebraic geometry.

Most of the results in the literature consider a pair $(X,\Delta)$ and draw conclusions on its lc centers. In this paper, we fix an ambient variety $X$ and a subvariety $Z$ and we try to answer the following:

\begin{qst}
Is there a divisor $\Delta$ such that $Z$ is a log canonical center for $(X,\Delta)$? What can be said about the pair $(X,\Delta)$?
\end{qst}

If such a divisor exists, we call $Z$ a \emph{potential} lc center for $X$.

As it turns out, quite a large class of subvarieties $Z$ are potential lc centers, at least with reasonable hypotheses on $X$. The first result is the following:
\begin{theoA}[Theorem \ref{ci}]
Let $X$ a $\Q$-Gorenstein normal variety and let $Y$ be any irreducible subvariety not entirely contained in $\Sing X$. Then there is a boundary $\Delta$ such that $Y$ is a log canonical center for the pair $(X,\Delta)$.
\end{theoA}

If we make stronger assumptions on the singularities of both the subvariety and the ambient variety, we can even classify the singularities of $(X,\Delta)$.

\begin{theoB}[Theorem \ref{lcpair}]
Let $X$ be a Gorenstein variety with lc singularities, and let $W$ be a reduced irreducible special complete intersection in $X$ (i.e. cut out by homogeneous forms of the same degree) with log canonical singularities, not contained in $\Sing X$. Then there exist a boundary $\Delta$ such that the pair $(X, \Delta)$ has log canonical singularities and $W$ is a log canonical center for $(X, \Delta)$.
\end{theoB}

This paper is structured as follows: the first section is devoted to the basic properties of lc centers. The second states and proves the first theorem. The third section is a brief introduction to the theory of Du Bois singularities, as they are used extensively in the fourth section to prove the second theorem.

\subsection*{Acknowledgements}
The author would like to thank his advisor S\'andor Kov\'acs for suggesting the problem and for his supervision. The author would also like to thank Alberto Chiecchio and Siddharth Mathur for numerous useful conversations.

\section{Background and notations}

\begin{defi}\label{lcdefi}
	Let $X$ be a normal variety over an algebraically closed field $k$ of characteristic zero, and let $\Delta$ be an effective Weil $\mathbb{R}$-divisor such that $K_X+\Delta$ is $\mathbb{R}$-Cartier. Assume furthermore that $\Delta$ is a subboundary, i. e. all of its coefficients are at most $1$. A subvariety $Z\subseteq X$ is called a \emph{log canonical center} for the pair $(X,\Delta)$ if:
	\begin{enumerate}
		\item The pair $(X,\Delta)$ has log canonical singularities at the generic point of every component of $Z$
		\item There is a birational morphism $f:Y\to X$ and a prime divisor $E$ in $Y$ such the discrepancy of $E$ with respect to $\Delta$ is $-1$ and $f(E) = Z$
	\end{enumerate}
\end{defi}

	If the pair $(X, \Delta)$ has log canonical singularities, more of the geometry of its lc centers is known.
	\begin{thm}[Theorem 5.14, \cite{kollar2013sing}]
	If $(X,\Delta)$ is a lc pair, then
	\begin{enumerate}
		\item Any intersection of lc centers is a union of lc centers
		\item Every union of lc centers is seminormal
		\item Every lc center that is minimal w. r. t. inclusion is normal
	\end{enumerate}
	\end{thm}

	In this paper, we will make use of the following
\begin{defi}
	Let $X$ be a variety in $\P^n$, and let $Y$ be a subvariety of $X$ whose codimension in $X$ is $r$. Then we say that $Y$ is a \emph{complete intersection} in $X$ if there are $r$ homogeneous forms $F_1,\ldots,F_r$ such that $Y$ is the subscheme theoretic of $X$ defined by the ideal $(F_1,\ldots,F_r)$. If $d_i = \deg F_i$, then we say that $Y$ is of type $(d_1,\ldots,d_r)$ in $X$.

	If $d_1 = \ldots = d_r$ we say that $Y$ is a \emph{special complete intersection} in $X$.
\end{defi}

	We will also need some well known results from commutative algebra.
\begin{defi}
	Let $A$ be a Notherian ring. We say that $A$ is $R_n$ if  for every prime $\p\subset A$ of height $\leq n$ the local ring $A_{\p}$ is regular.
	We say that $A$ is $S_n$ if  for every prime $\p\subset A$ of height $\leq n$ we have $\depth A_{\p}\geq n$.
\end{defi}
\begin{prop}[Serre's criterion, \cite{eisenbud1995commutative}]
	A ring A is normal iff it's $R_1$ and $S_2$. A ring $A$ is reduced iff $R_0$ and $S_1$.
\end{prop}

	Finally, we will use the following notation: if $A$ is a graded ring and $I$ a homogeneous ideal, $V(I)$ will denote the closed subscheme of $\text{Proj} A$ cut out by the ideal $I$.

\section{Log canonical centers of $\Q$-Gorenstein varieties}
	The main result of this section is:
\begin{thm} \label{ci}
	Let $X$ a $\Q$-Gorenstein normal variety and let $Z\subseteq X$ be any irreducible subvariety not entirely contained in $\Sing X$. Then there is a boundary $\Delta$ such that $Y$ is a log canonical center for the pair $(X,\Delta)$.
\end{thm}
	Note that Theorem \ref{ci} implies that any subvariety of a smooth variety is a log canonical center.

	The structure of the proof can be summarized as follows: first, we will embed $Z$ in a complete intersection $W$ and consider the birational morphism $f: Y\to X$, where $Y$ is the normalization of the blow up of $X$ at $W$. Then we will produce a divisor $\Delta$ such that the unique component of the exceptional divisor of $f$ that dominates $Z$ has discrepancy $-1$ for the pair $(X,\Delta)$.

	We will begin by reducing to the case where $X$ is Cohen-Macaulay, as being an lc center is a generic property. More precisely: let $(X,\Delta)$ be a pair and let $U\subseteq X$ be an open set. If $Z$ is an irreducible subvariety of $X$ with $U \cap Z \neq \emptyset$ and $U \cap Z$ is an lc center for $(U, \Delta\left|_U \right.)$ then it is easy to see that $Z$ is an lc center for $(X, \Delta)$.

	We can also assume that where $Z$ is a complete intersection in $X$.
\begin{lm}\label{yayci}
	Every subvariety $Z$ of a Cohen-Macaulay variety $X$ is a component of a reduced complete intersection in $X$.
\end{lm}

\begin{proof}
	Let $S$ be the homogeneous coordinate ring of $X$ in $\P^n$. The ideal $I_Z \subseteq S$ is prime, hence we can apply Lemma \ref{ciinideal} to get a (homogeneous) non zerodivisor $g_1\in I_Z$ such that $S/g_1$ is reduced and CM. Continuing in this fashion one gets a regular sequence $g_1,\ldots,g_r$ contained in $I_Z$, where $r = \textrm{codim }Z$. Since the ideal  is radical and equidimensional, then it's the intersection of the minimal prime ideals containing it. By construction, the height of $I_Z$ is the same as the height of $(g_1, \ldots, g_r)$, so $I_Z$ has to be one of its minimal primes, hence $Z$ is a component of the complete intersection cut out in $X$ by the $g_i$'s.
\end{proof}

\begin{lm}\label{ciinideal}
	Let $R$ be a finite type reduced graded Cohen-Macaulay $k$-algebra, with $\newchar k = 0$, and let $\p$ be a prime in $R$ of height $\geq 1$. Then there is a homogenous element $g\in \p$ such that $g$ is a non zerodivisor and $R/g$ is reduced and CM.
\end{lm}

\begin{proof}
	Write $R = S/I$ where $S$ is a graded polynomial ring over $k$, and let $f_1,\ldots, f_N$ be homogeneous polynomials of the same degree $d$ such that $\p_{\geq d}$ is generated by the images of $f_i$'s in $S/I$.
	The idea is to look for an element of the form
\[
	g = \sum\lambda_i f_i,\quad\lambda_i\in k
\]
	Let's first establish that at least a non zerodivisor of that form exists: if all linear combinations were zero divisors, then they would lie in the union of the minimal primes $\p_i$ of $R$, thus implying that $\p_{\geq d}\subseteq \bigcup \p_i$. If $a\in\p$ then for some $q$ we would have $a^q\in \p_{\geq d}$, thus $a^q\in \bigcup \p_i$, which in turn forces $a\in \bigcup\p_i$. This would show that $\p\subseteq\bigcup\p_i$ and by the prime avoidance lemma then $\p\subseteq\p_j$ for some $j$, which is impossible by the assumption on the height of $\p$: therefore at least one of the $f_i$'s is not a zero divisor.

	Now we shall show that there is a non empty open set $U\subseteq k^N$ such that $g$ is a non zerodivisor in $R$ for all $(\lambda_i)\in U$. Indeed, $g$ is a zerodivisor iff it is contained in one of the minimal primes $p_i$'s, and this happens precisely when $(\lambda_i)$ is in the kernel of the $k$-linear map
	\[
		\psi_i: k^N\to R_{\p_i}/\p_i R_{\p_i}, (\lambda_i)\mapsto g
	\]
By the reasoning in the previous paragraph, the kernel of $\psi_i$ is not the whole of $k^N$, and setting $U = k^N\setminus \bigcup \ker \psi_i$ yields the desired open set.

	Since $\newchar  k = 0$, then by Bertini's theorem the generic element $g$ of the linear system generated by the $f_i$'s is smooth on the nonsingular locus of $Z(I)$ away from $Z(\p)$, and from the discussion above $g$ is also a non zerodivisor. In particular, $R/g$ is reduced at the generic point of each component, hence it's $R_0$. Since $R$ is CM and $g$ a non zerodivisor, we have that $R/g$ is CM too: in particular it's $S_1$. It follows that $R/g$ is reduced.
	Finally, note that since $R$ is CM and $g$ is a non zerodivisor, then $R/g$ is CM as well.
\end{proof}

Note that this implies that every subvariety of $\P^n$ is a component of a reduced complete intersection, a fact that is quite interesting on its own.

Now we will proceed to compute the mutiplicities of the exceptional divisors of the blow up of $X$ at a complete intersection.

\begin{lm}\label{cicomponents}
	Let $X$ be $\Q$-Gorenstein and CM, and let $W$ be a complete intersection in $X$. Write $W = W_{1}\cup\ldots\cup W_{q}$ as the union of its irreducible components, and let $Y$ be the normalization of the blow up of $X$ at $W$. Then the exceptional divisor $E$ has exactly $r$ components $E_{i}$ and moreover we have
	\[
		K_{Y} \sim f^{*}K_{X} + (r-1) E_{1} + \ldots + (r-1)E_{q}
	\]
\end{lm}
 
\begin{proof}
	Let $Y'$ be the blow up of $X$ at $W$, with exceptional locus $E$. Since $W$ is a complete intersection, then the ideal of $W$ in $X$ is locally generated by a regular sequence, as one can easily check. Therefore, the very same proof of \cite[Theorem II.8.24(b)]{hart77} gives that the exceptional divisor is a projective bundle over $W$, hence $E$ has exactly as many components as $W$. It follows that the preimage $V$ of the smooth locus of $W$ is smooth in $Y'$: in particular $Y'$ is regular in codimension one.
	
	Let $n: Y\to Y'$ be the normalization map and let $f: Y\to X$ be the composition of $Y$ with the blow up map. The R1 condition on $Y'$ implies that $n$ is an isomorphism in codimension one, therefore the exceptional locus has codimension at least $2$ and $n: n^{-1}E\to E$ induces an isomorphism over the smooth locus of each component $E_i$ of $E$. Therefore, in order to compute the discrepancies of the $E_i$'s we can restrict to looking at the smooth locus of $W_i$. Let $U_i$ be a neighborhood of the smooth locus of $W_i$ in $X$ on which $X$ is smooth. By \cite[Ex. II.8.5]{hart77} we have that
	\[
		K_{Y}  \left|_{f^{-1}U_i}\right.\sim {f^{*}K_{X}}\left|_{f^{-1}U_i}\right. + {(r-1)E_i}\left|_{f^{-1}U_i}\right.
	\]
	Since $U_i$ cannot contain any other components of $W$ other than $W_i$, this concludes the proof.
\end{proof}
 
Now we need to produce a boundary $\Delta$: if $D_i$ denotes the divisor determined by $F_i$, then a natural choice is to set $\Delta = D_1+\ldots +D_r$. It turns out that this actually works. First of all we check that $(X,\Delta)$ satisfies the condition (1) of the definition of lc center.

\begin{lm}\label{sncdelta}
	The pair $(X,\Delta)$ is snc away from the union of the singular loci of $W$ and $X$.
\end{lm}

\begin{proof}
	Let $U = X \setminus \left(\Sing X \cup \Sing W\right)$, and let $p\in U$. We will first prove that every component $D_i$ of $\Delta$ is smooth at $p$.

	Let $(G_1,\ldots,G_s)$ be the ideal of $X$ in $\P^n$, and assume that $W$ is cut out by the regular sequence $F_1,\ldots,F_r$. Then the Jacobian of $W$ as a subvariety of $\P^n$ is given by
	\[
	J = \left[
			\begin{array}{c}
  				\nabla F_1  \\
  				\ldots \\
  				\nabla F_r   \\
				\nabla G_1  \\
  				\ldots \\
  				\nabla G_s  
			\end{array}
		\right]
	\]
	where the zero locus of $F_i$ is the component $D_i$ of $\Delta$.
	
	Let $r' = n - \dim X$: since $X$ is smooth, then the rank of the submatrix of $J(p)$ obtained by considering the last $s$ rows is $r'$. If any of the first $r$ rows were zero, then the rank of $J(p)$ would be less than $r+r'$, and this would contradict the smoothness of $p$. Therefore $\nabla G_i \neq 0$ for all $i$'s: this shows that the divisors $D_i$ are smooth at $p$.
	
	We now claim that the $D_i$'s intersect transversely at $p$: this amounts to showing that 
		\[
			\textrm{codim }_{T_p X} (T_pD_1\cap\ldots\cap T_pD_r \cap T_p X) = \codim_{X,p}(D_1) + \ldots + \codim_{X,p}(D_r)
		\]
		
	Observe that each term in the right hand side of the above equality is $1$, hence the r.h.s. is equal to $r$. On the other hand, since $T_pD_1\cap\ldots\cap T_pD_r \cap T_p X$ is the tangent space to $W$ at $p$, counting (co)dimensions yields that the l.h.s. is equal to $r$, too.
	\end{proof}

Now we compute the pullbacks of the components of $\Delta$. The idea is that the fact that they meet transversely forces each $D_i$ to contain each component of $W$ with multiplicity one. More precisely, we have the following

\begin{lm} \label{snc}
	Let $\tilde{D_{i}}$ the proper transform of $D_{i}$. Then
		\[
			f^{*}(D_{i}) \sim \tilde{D_{i}} + E_{1}+\ldots + E_{q}
		\]
\end{lm}

\begin{proof}
	The proof of Lemma \ref{cicomponents} implies that the $E_i$'s are the only exceptional divisors of $f$. Therefore
	\[
	f^{*}(D_{i}) \sim \tilde{D_{i}} + \mu_1E_{1}+\ldots + \mu_r E_{q}
	\]
	for some $\mu_i\in \ZZ$. The idea is to compute the $\mu_j$'s locally on the the smooth locus of the $W_j$ and to show that they are all equal to $1$.
	
	We also know from Lemma \ref{sncdelta} that every $D_{i}$ is smooth on an open subset $U$ of $W_{j}$. Let's then work locally and assume that $W\subset X =\sp R $ is a smooth integral complete intersection given by the ideal $I = (f_1,\ldots,f_r)$ of the ring $R$ and $D_i = D$ is the zero set of an element $g\in I$, with $D$ smooth along $W$.
	
	With this setup, the blow up $Y = Bl_W X \subset X \times \P^{r-1}$ is given by taking Proj of the graded ring
	\[
	S = R[t_1,\ldots,t_r]/(t_if_j - t_jf_i  | i,j=1,\ldots,r)
	\]
	and the ideal $J$ of the exceptional divisor $E$ is the ideal generated by $I$ in $S$. We want to show that $g\in J\setminus J^2$, where $g$ is the image of $g$ under the natural injection $R\hookrightarrow S$, as this will ensure that $f^{*}(D) \sim \tilde{D} + E$.
	
	Assume not, i.e. suppose we can write
	\[
	g = \sum \alpha_{ij}(x,t) f_i f_j, \quad \alpha_{ij}(x,t)\in S
	\]
	Then this would imply
	\[
	g = \sum \alpha_{ij}(x,0) f_i f_j, \quad \alpha_{ij}(x,0)\in R
	\]
	thus $g\in I^2$, which contradicts the assumption that $D$ is smooth along $W$.
	\end{proof}
	
	Putting everything together, we can prove the main result of this section.

\begin{proof}[Proof of Theorem \ref{ci}]
	By using Lemma \ref{yayci} we can assume that $Z$ is a component of a complete intersection $W$ in $X$, so $Z = W_i$ for some $i$. 
	Using the same notation as before, with $\Delta = D_{1}+\ldots + D_{r}$ we get
	\begin{eqnarray*}
		K_{Y} - f^{*}(K_{X}+\Delta) & \sim & (r-1) \sum E_{i} - f^{*}\Delta\\
							 & \sim & (r-1) \sum E_{i} - \sum \tilde{D_{i}} - r\sum E_{i}\\
							 & \sim & - \sum \tilde{D_{i}} - \sum E_{i}
	\end{eqnarray*}
from which it follows that the discrepancy of $E_i$ with respect to $(X, \Delta)$ is $-1$ for all $i$'s. This implies that every component $W_{i}$ of $W$ satisfies the condition $(2)$ of Definition \ref{lcdefi} for the pair $(X,\Delta)$. Moreover, $(X,\Delta)$ satisfies (1) by Lemma \ref{snc}.
\end{proof}

It is worth to point out that with these hypotheses on $X$ and $Y$ the pair $(X,\Delta)$ might not be lc at every point of $X$, since if the pair $(X,\Delta)$ is lc then every log canonical center has to be seminormal and Du Bois (see \cite[Section 4.3 and Theorem 5.14]{kollar2013sing}).

\section{Du Bois singularities}
\label{DB}
	If $X$ is a smooth variety, then one can construct a special resolution of the constant sheaf, namely the DeRham complex. It turns out that an analogue construction can be carried out for an arbitrary normal variety $X$ over $\mathbb{C}$, yielding the \emph{Deligne-Du Bois} complex: it is an object $\underline{\Omega}^\bullet_X$ in $D_{coh}^\flat(X)$ that has a natural map $\mathcal{O}_X \to \underline{\Omega}^\bullet_X$. If this last morphism is a quasi-isomorphism, then we say that $X$ has \emph{Du Bois singularities}.
	
	The technical details of the definition can be found in \cite[Chapter 6]{kollar2013sing}: here we state Schwede's equivalent definition (\cite{schwede2007simple}).
	
	\begin{defi}
	Let $X$ be normal and let $X\subset Y$ be an embedding in a smooth scheme. Let $f: Z\to Y$ be a log resolution that is an isomorphism outside of $X$, and let $E$ be the pre image with the reduced induced subscheme structure. Then $X$ has \emph{Du Bois} singularities if the natural map $\mathcal{O}_{X}\to Rf_{*}\mathcal{O}_{E}$ is a quasi-isomorphism. 
	\end{defi}
	
	We are interested in Du Bois singularities since in many cases of interest they are closely related to log canonical singularities.

	In one direction we have:
	\begin{thm}[\cite{kovacs1999rational}]
	Let $X$ be a normal variety with DB singularities such that $K_X$ is Cartier. Then $X$ is log canonical.
	\end{thm}

	Conversely,
	\begin{thm}[\cite{kollar2010log}, Theorem 1.4]
	Let $(X,\Delta)$ be a log canonical pair. Then $X$ is Du Bois.
	\end{thm}

	We now state a few results that will be used in the next section. The first one is a deformation result:

\begin{thm}[\cite{ks2011bois}, Theorem 4.1]
Let $X$ be a scheme of finite type over $\mathbb{C}$ and $H$ a reduced Cartier divisor on $X$. If $H$ has Du Bois singularities, then $X$ has Du Bois singularities near $H$.
\end{thm}

	A substantial advantage of Du Bois singularities over lc singularities is the notion of a Du Bois pair. Given any reduced subscheme $Z\subseteq X$, there's a natural map $\underline{\Omega}^\bullet_X \to \underline{\Omega}^\bullet_Z$, that can be completed to a distinguished triangle
	\[
	\underline{\Omega}^\bullet_{Z,X} \to\underline{\Omega}^\bullet_X \to \underline{\Omega}^\bullet_Z\stackrel{+}{\to}
	\]
	and one can prove that there is a natural map $\mathcal{I_Z}\to \underline{\Omega}^\bullet_{Z,X}$.
	
	\begin{defi}
	A pair $(X,Z)$ is a \emph{Du Bois pair} if $\mathcal{I_Z}\to \underline{\Omega}^\bullet_{Z,X}$ is a quasi-isomorphism.
	\end{defi}
	
	One would expect that if both $X$ and $Z$ are Du Bois then the pair $(X,Z)$ is, too. This is true, by the following very useful
	\begin{lm}[\cite{kollar2013sing}, Prop. 6.15]
	Let $(X,Z)$ be a reduced pair. If two of $\{ X,Z,(X,Z) \}$ are Du Bois, so is the third.
	\end{lm}
	
	Another effective way to decide whether a pair is Du Bois or not, is the next
	
	\begin{lm}[\cite{kollar2013sing}, Excision lemma, Proposition 6.17\phantom{}]
Let $X = (Y\cup Z)_{red}$ be the union of closed reduced subschemes with $W = (Y\cap Z)_{red}$. Then $(X,Y)$ is a Du Bois pair iff $(Z,W)$ is a Du Bois pair.
\end{lm}
	
	One might wonder about the relationship between being a Du Bois pair and being an lc pair. The following theorem provides an answer. 

	\begin{thm}[\cite{graf2014potentially}, Theorem 1.4]
	Let $X$ be a nomal complex variety. Let $(X,\Sigma)$ be a Du Bois pair and $\Delta$ a reduced effective divisor on $X$ such that $	\textrm{supp}\Delta\subseteq \Sigma$ and $K_X+\Delta$ is Cartier. Then $(X,\Delta)$ is a log canonical pair.
	\end{thm}
	
	The theorems above will be used extensively in the next section. Meanwhile, as an example of an immediate application, one can see that a normal lc center $Z$ of $X$ such that $K_Z$ is Cartier has lc singularities, as any union of lc centers is Du Bois (\cite[Thm. 5.14]{kollar2013sing}).

\section{A global lc condition for special lc complete intersections}

	The aim of this section is to show that under certain hypotheses one can prove the existence of a \emph{log canonical} pair $(X, \Delta)$ such that $Z$ is a log canonical center for $(X, \Delta)$.
	
	The precise statement is the following:
\begin{thm}\label{lcpair}
	Let $X$ be a normal variety with lc singularities such that $K_X$ is Cartier, and let $W$ be a reduced irreducible special complete intersection in $X$ such that $W$ has lc singularities and it is not contained in $\Sing X$. Then there exist a boundary $\Delta$ such that the pair $(X, \Delta)$ has log canonical singularities and $W$ is a lc center for $(X, \Delta)$.
\end{thm}

	Inspired by the proof of Proposition \ref{ci}, the idea is to produce new generators for the ideal of $I_W$, then use Bertini's theorem to control the singularities of their zero sets and finally to rely on results from the theory of Du Bois pairs to draw conclusions on the pair $(X, \Delta)$.
	
	We will break down the proof in a few technical lemmas. For the rest of the section, we will employ the following
	
	\begin{setup}\label{setup}
	The variety $W$ will be a reduced complete intersection in $X$ given by the ideal $I_W = (F_1,\ldots,F_r)\subseteq S(X)$, where $S(X)$ is the coordinate ring of $X$ in $\P^n$.	
	\end{setup}
	
	Since we'll be using partial complete intersections in inductive arguments, we start by stating a lemma about reducedness:
\begin{lm}\label{veryreduced}
	Using the notation of Setup \ref{setup}, for any $s \leq r$ the partial complete intersection $(F_1,\ldots,F_s)$ is reduced.
\end{lm}

\begin{proof}
It's enough to show that $J = (F_1,\ldots,F_{r-1})$ is reduced. Assume towards contradiction that there is a homogeneous element $G\in S(X)$ such that $G\notin J$ but $G^q\in J$ for some $q>0$. Among all such $G$'s, pick one with the minimal degree. By hypothesis, $G^q\in J\subset I$ so $G\in I$. We can write
	\[
	G = H +M F_r
	\]
	where $H$ and $M$ are homogeneous elements with $H\in J$ and $\deg M <\deg G$. Then
	\[
	0 \equiv G^q \equiv M^q F_r^q \mod J
	\]
	Since the $F_i$'s form a regular sequence, $F_r^r$ is not a zerodivisor modulo $J$, hence $M^q\in J$. By minimality of the degree of $G$, $M$ has to be an element of $J$, and this forces $G\in J$, contradiction.
\end{proof}

	Many proofs rely on the fact that $I_W$ is generated by a regular sequence. Since we want to consider a (possibly) different set of generators, we have to make sure they still form a regular sequence: this turns out to be the case.

\begin{lm}\label{genreg}
	Using the notation of Setup \ref{setup}, every generating set $G_1,\ldots, G_r$ for $I_W$ consisting of $r$ elements is a regular sequence.
\end{lm}

\begin{proof}
	Let $J = (G_1,\ldots,G_{s-1})$ and assume towards contradiction that $G_s$ is a zerodivisor modulo $J$. Then $G_s$ is contained in a minimal prime of $J$, which are the same as the minimal primes of $\sqrt{J}$. This implies that $V(G_s)$ contains an irreducible component of $V(J)$, so that $\dim V(J+(G_s)) = \dim V(J)$. Since intersecting with a hypersurface makes the dimension go down by at most one, it would follow that $\dim V(I_{W}) > n-r$, contradiction.
\end{proof}

	Before proving the main lemma, we need a way to control the singularities of a complete intersection. 
\begin{lm} \label{regularcuts}
	Let $(R,\m)$ be a local Noetherian ring, $f\in \m$ a non zero-divisor and assume that $(R/f, \m/f)$ is regular. Then $(R,\m)$ is a regular ring.
\end{lm}

\begin{proof}
Let $n=\dim R$. Then since $f$ is a regular element $\dim R/f = n-1$, hence there are $n-1$ elements $a_i\in \m$ such that
	\[
		(a_1,\ldots,a_{n-1}) = \m/f\textrm{ in }R/f
	\]
	hence 
	\[
		(a_1,\ldots,a_{n-1}, f ) = \m
	\]
	This proves that $\m$ can be generated by $n$ elements, hence $R$ is regular.
\end{proof}

	Now we are ready to state and prove the main technical tool of this section: basically, it says that a complete intersection can be reached by a sequence of complete intersections with mild singularities.

\begin{lm} \label{genericgens}
	Using the notation of Setup \ref{setup}, and assume furthermore that all the degrees of the $F_{i}$'s are the same. Then there is a set of generators $G_{1},\ldots, G_{r}$ for $I_{W}$ such that for every subset $S\subseteq \{1,\ldots,r\}$ the complete intersection \[D_S=\bigcap_{i\in S}V(G_{i})\] is reduced, normal and smooth away from the singular locus of $W$.

\end{lm}

\begin{proof}
	Let $G_i = \sum \mu_{ij} F_j$ where $\mu_{ij}\in k$. Let $U\subset \A^{r^2}$ be the open subset where the matrix $\Xi = (\mu_{ij})$ is invertible. It's easy to check that for any choice of $(\mu_{ij})\in U$ the $G_i$'s generate the ideal $I_W$, so they form a regular sequence by Lemma \ref{genreg}.
	
	It follows that the $D_S$'s are complete intersections, and their reducedness follows directly from Lemma \ref{veryreduced}.
	
	Let $D_{i} = V(G_{i})$. We claim that there is an open set $V\subset \A^{r^2}$ such that all the intersections
	\[
		D_1, D_1\cap D_2,\ldots, D_1\cap\ldots\cap D_{r-1}
	\]
	are smooth away from $\Sing W$. Note that by applying the same reasoning to any permutation of the set $\{1,\ldots,r\}$ and intersecting all the $V$'s thusly obtained we can prove the claim for an arbitrary $D_S$.
	
	Consider the linear system $\mathfrak{d}$ on $X$ generated by the $F_i$'s: by Bertini's theorem (cfr. \cite[Rem. III.10.9.2]{hart77}), for a generic choice of the first row of $\Xi$ the scheme $V(G_1)$ is smooth away from $W$ and $G_1,F_2,\ldots$ generate $I_W$. Let $w\in W$ be a smooth point: then by Lemma \ref{regularcuts}, it's also a smooth point of $V(G_1,\ldots,F_{r-1})$. By applying the same lemma repeatedly, we get that $D_1$ is smooth at $w$.
	
	Having chosen $G_1$, by restricting $\mathfrak{d}$ to $D_1$ we can apply Bertini's theorem again to get that for a generic choice of the second row of $\Xi$ the zero set of $G_2$ is smooth away from $\Sing D_1$. Continuing in this fashion we get the desired open set $V$.
	
	Moreover each $D_S$ is a complete intersection that is smooth away from $\Sing W$, whose codimension in $D_S$ is at least $1$ since $W$ is normal. Therefore every $D_S$ is normal.
\end{proof}
Note that the fact that every $D_S$ is normal implies that it is actually irreducible, as complete intersections are connected.
We can even say more about the singularities of $D_S$.

\begin{lm} \label{DB gens}
	Let $W$ be an irreducible reduced complete intersection in $X$ with lc singularities and let $G_1, \ldots, G_r$ be generators for $I_W$ such that each partial complete intersection $D_S$ is smooth away from $W$. 
	Then the pair
	\[
	(D_S, D_{S'}),\quad\textrm{ where }S' = S \cup\{j\}\textrm{ for some }j\notin S  
	\]
	is a Du Bois pair.
\end{lm}
Note that the $H_s$'s are reduced by Lemma \ref{veryreduced}, so the pair $(H_s, H_{s+1})$ is actually a reduced pair.
\begin{proof}
	The idea is to use descending induction on $|S|$.
	Assume that $D_{S'} =D_S \cap D_j$ has Du Bois singularities: then $D_S$ has Du Bois singularities near $W$ by \cite[Theorem 4.1]{ks2011bois}. Moreover, $H_s$ is smooth away from $W$, therefore actually $D_S$ has Du Bois singularities.
	By \cite[Prop. 6.15]{kollar2013sing} then the pair $(D_S, D_{S'})$ is Du Bois: in particular $D_S$ has Du Bois singularities.
	
	Now observe that since $D_S = W$ has lc singularities by hypothesis when $|S| =r$, then by \cite[Theorem 1.4]{kollar2010log} $D_S$ is Du Bois when $|S|=r$.
	
\end{proof}
Note that this proof implies that all the $D_S$'s are Du Bois, hence if $S'$ is \emph{any} superset of $S$, every pair $(D_S, D_{S'})$ is Du Bois.

An interesting corollary of the proof is that the pairs $(X, D_i)$ are lc by \cite[Thm. 1.4]{graf2014potentially}, as both $K_X$ and $D_i$ are Cartier. It turns out that we can exploit this fact to get a stronger result.

\begin{lm} \label{excision}
	Assume $Y$ is reduced, Du Bois and $B_1,\ldots,B_r$ are reduced, Du Bois Cartier divisors such that for all $q\leq r$ every possible intersection
	\[
	 B_{s_1}\cap\ldots \cap B_{s_q}
	\]
	is reduced and Du Bois. Then the pair $(Y,  B_1\cup\ldots\cup B_r)$ is Du Bois.
\end{lm}

\begin{proof}
	The idea is to use induction and the excision lemma for Du Bois singularities (see \cite[Prop. 6.17]{kollar2013sing}).
	If $r=1$, the statement is clearly true, since $Y$ Du Bois and $B_1$ Du Bois implies that the pair $(Y, B_1)$ is Du Bois as well.
	Assume the claim is true for $r-1$. It's enough to show that $B_1\cup\ldots\cup B_r$ is Du Bois, since we already know that $Y$ is Du Bois. By the induction hypothesis we know that $B_1\cup\ldots\cup B_{r-1}$ is Du Bois. Then the claim will follow if we can prove that
	\[
	(B_1\cup\ldots\cup B_r, B_1\cup\ldots\cup B_{r-1})
	\]
	is a Du Bois pair. By the excision lemma, this is equivalent to showing that
	\[
	(B_r,( B_1\cap B_r) \cup\ldots \cup (B_{r-1}\cap B_r))
	\]
	is Du Bois, and this last claim follow from the inductive hypothesis, as $B_r$ is reduced and Du Bois and $B_j\cap B_r$ is a reduced Du Bois Cartier divisor on $B_r$ for each $j$.
\end{proof}

\begin{proof}[Proof of Theorem \ref{lcpair}]
	By Lemmas \ref{genericgens} and \ref{DB gens}, there is a sequence of generators $G_i$ for $I_W$ such that all the possible partial intersections have Du Bois singularities. Then Lemma \ref{excision} yields that the pair
	\[
	(X, \bigcup D_i)
	\]
	is Du Bois. If $\Delta  = D_1+\ldots + D_r$ then, since $K_X$ is Cartier by \cite[Thm. 1.4]{graf2014potentially} we have that $(X,\Delta)$ is a log canonical pair. The other part of the statement follows from Proposition \ref{ci}.
\end{proof}

\begin{rmk}
	Actually the proof of this last theorem shows that if a subvariety $Z$ of $X$ is a component of a special Du Bois complete intersection, then $Z$ is a lc center for a \emph{log canonical} pair $(X,\Delta)$.
\end{rmk}

	\providecommand{\bysame}{\leavevmode\hbox to3em{\hrulefill}\thinspace}
\providecommand{\MR}{\relax\ifhmode\unskip\space\fi MR }
\providecommand{\MRhref}[2]{%
  \href{http://www.ams.org/mathscinet-getitem?mr=#1}{#2}
}
\providecommand{\href}[2]{#2}

\end{document}